\documentclass[11pt, reqno]{amsart}
\usepackage{amsmath, amssymb, amsthm, graphicx,marvosym}
\usepackage{cancel}
\usepackage[nobysame]{amsrefs}[11pts]

\newtheorem{theorem}{Theorem}[section]
\newtheorem{proposition}[theorem]{Proposition}
\newtheorem{lemma}[theorem]{Lemma}

\newtheorem{definition}[theorem]{Definition}

\makeatletter \@addtoreset{equation}{section} \makeatother

\newcommand{\pp}{\partial_+}
\newcommand{\pn}{\partial_-}
\def\tilde{\widetilde}

\newcommand{\ve}{\varepsilon}

\newcommand{\beq}{\begin{equation}}
\newcommand{\eeq}{\end{equation}}

\def\TS{\textstyle}
\def\com#1{\quad{\textrm{#1}}\quad}
\def\eq#1{(\ref{#1})}
\def\nn{\nonumber}

\def\dbyd#1{\TS{\frac{\partial}{\partial#1}}}
\begin{document}

\title[Optimal density lower bound on nonisentropic gas dynamics]{Optimal density lower bound on nonisentropic gas dynamics}

\author{Geng Chen}
\address{Department of Mathematics, University of Kansas, Lawrence, KS 66045, U.S.A. ({\tt  gengchen@ku.edu}).}

\date{\today}

\begin{abstract}
{\small
In this paper, we prove a time dependent lower bound on density in the optimal order $O(1/(1+t))$ for the general smooth nonisentropic flow of compressible Euler equations.
}

\end{abstract}
\maketitle

2010\textit{\ Mathematical Subject Classification:} 76N15, 35L65, 35L67.


\section{Introduction}
The compressible Euler equations, which is widely used to describe the inviscid compressible fluid such as gas dynamics, in Lagrangian coordinates in one space dimension, satisfy
\begin{align}
\tau_t-u_x&=0\,,\label{lagrangian1}\\
u_t+p_x&=0\,,\label{lagrangian2}\\
\textstyle\big(\frac{1}{2}u^2+e\big)_t+(u\,p)_x&=0\,, \label{lagrangian3}
\end{align}
where $\rho$ is the density, $\tau=\rho^{-1}$ is the specific volume,
$p$ is the pressure, $u$ is the velocity, $e$ is the  specific  internal energy, $t\in\mathbb{R}^+$ is the time
and $x\in\mathbb{R}$ is the spatial coordinate.  
The system is closed by the
Second Law of Thermodynamics:
\beq
  T\,dS = de + p\,d\tau,
\label{2TD}
\eeq
where $S$ is the entropy and $T$ the temperature.
In this paper, we consider the polytropic ideal gas, in which
\[
  p\,\tau=R\,T   \com{and}  e=c_v\,T=\frac{p\tau}{\gamma-1}
\]
with ideal gas constant $R>0$, and specific heat constant $c_v>0$,  which  implies 
\beq
   p=K\,e^{\frac{S}{c_v}}\,\tau^{-\gamma} \com{with adiabatic gas constant} \gamma>1\,,
\label{introduction 3}
\eeq
where 
$K$ is a positive constant, c.f. \cite{courant} or \cite{smoller}.  The classical solutions for compressible Euler equations in Lagrangian and Eulerian coordinates are equivalent \cite{Dafermos2010}. 

For any $C^1$ solution,
it follows that (\ref{lagrangian3}) is equivalent to the conservation of entropy \cite{smoller}:
\beq
   S_t=0\,,
\label{s con}
\eeq
hence 
\[S(x,t)\equiv S(x,0)=: S(x).\]
If the entropy is constant, the flow is isentropic, then
(\ref{lagrangian1}) and (\ref{lagrangian2}) become a closed system,
known as the $p$-system.
%\begin{align}
%\tau_t-u_x&=0\,,\label{p1}\\
%u_t+p_x&=0\,,\label{p2}
%\end{align}
%with
%\beq\label{p3}
%   p=K\,\tau^{-\gamma}\,,\qquad \gamma>1,
%\eeq
%where, without loss of generality, we still use $K$ to denote the constant in pressure.

%In this paper, we consider the classical solutions of initial value problems for full Euler equations
%(\ref{lagrangian1}), (\ref{lagrangian2}), (\ref{introduction 3}) and (\ref{s con}) with initial data
%$\big(u(x,0), \tau(x,0), S(x,0)\big)$ and isentropic Euler equations (\ref{p1})$\sim$(\ref{p3}) with 
%initial data $\big(u(x,0),  \tau(x,0)\big)$. 
%We consider the large data problem, which means that there is no restriction on the size of the solutions.

In order to establish a large data global existence theory for BV solutions
of compressible Euler equations, which is a major open problem in the field of
hyperbolic conservation laws, we need a sharp estimate on the possible time decay on density. In fact, a solution loses its strict hyperbolicity as density approaches zero when $t\rightarrow \infty$. And one cannot expect to find a constant positive lower bound on $\rho(x,t)$ in general even when the initial density has one, see examples in \cite{courant, jenssen} which will be introduced later. Also see \cites{ls} for the difficulty in proving the global existence of BV solution when solution is close to the vacuum.

%Furthermore, the time-dependent lower bound on density for classical solutions can be used to study the and life-span of classical solutions. 

The study on the lower bound of density for classical solutions can be traced back to Riemann's pioneer paper \cites{Riemann} in 1860, in which he considered a special isentropic wave interaction between two strong rarefaction waves. Using Riemann's construction, a Lipschitz continuous example including an interaction of two centered rarefaction waves for isentropic Euler equations
was provided in Section 82 in \cites{courant}, in which the function $\min_{x\in\mathbb{R}}\rho(x,t)$ was proved to decay to zero in an order of $O(1+t)^{-1}$ as $t\rightarrow\infty$, while the initial density is uniformly away from zero,
when the adiabatic constant $\gamma=\frac{2\mathcal{N}+1}{2\mathcal{N}-1}$ with any positive integer $\mathcal{N}$. This result was extended to any $\gamma>1$ in \cite{jenssen}. Nonisentropic shock-free examples with density approaching zero in infinite time can be found in \cite{G6}.

Then one natural goal is to show that for any classical solution $\min_{x\in\mathbb{R}}\rho(x,t)$ has a lower bound in the optimal order of $O(1+t)^{-1}$, under the assumption that initial density is uniformly positive.
For isentropic rarefactive solutions, for any $\gamma>1$,  Long Wei Lin first proved that the density has a lower bound in the order of  $O(1+t)^{-1}$ in \cite{lin2} by introducing a polygonal scheme. Next, for general classical isentropic and nonisentropic solutions (possibly including compression), in  \cite{CPZ}, the author, Pan and Zhu found a lower bound on density in the order of  $O(1+t)^{-4/(3-\gamma)}$ when $1<\gamma<3$. 

For the general isentropic classical solution with any $\gamma>1$, the time-dependent lower bound on density in the optimal order $O(1+t)^{-1}$ was found by the author in \cite{G9}. The key idea is to find an invariant domain on some gradient variables measuring backward and forward rarefactions. For any nonisentropic classical solution, the lower bound on density in an almost optimal order $O(1+t)^{-1-\delta}$ for any $0<\delta\ll 1$ was also proved in \cite{G9}.

In this paper, we will prove that, for any classical nonisentropic solution, the density indeed has a lower bound in the order of $O(1+t)^{-1}$. 
\begin{theorem}\label{thm1}
Consider a $C^1$ solution $\big(u(x,t), \tau(x,t), S(x)\big)$ 
of the initial value problem for the compressible Euler equations \eqref{lagrangian1}-\eqref{lagrangian2}, \eqref{s con} and \eqref{introduction 3},
in the region $(x,t)\in \mathbb R\times[0,T)$, with initial data $\big(u(x,0), \tau(x,0)>0, S(x)\big)\in C^1\cap L^\infty(\mathbb R)$. Here, $T$ can be any finite positive constant or infinity. Furthermore, assume that the initial data
$\rho(x,0)=1/\tau(x,0)$, $S'(x)$, $\tilde\alpha(x,0)$ and $\tilde\beta(x,0)$ are all uniformly bounded for any $x\in \mathbb{R}$, where $\tilde\alpha$ and $\tilde\beta$ are defined later in \eqref{al_def0}-\eqref{be_def0}. The total variation of $S(x)$ is finite. Then there exists a constant $C_1$ only depending on the $C^1\cap L^\infty$ norm of initial data
$(u,\tau,\rho,S)(x,0)$, but independent on $T$, such that,
\beq\label{max}\min_{x\in\mathbb{R}}{\rho(x,t)}>\frac{C_1}{1+t}\,.
\eeq
And  $\tilde\alpha(x,t)$ and $\tilde\beta(x,t)$ are both bounded above by constants independent on $T$ when  $(x,t)\in \mathbb R\times[0,T)$.
\end{theorem}
This lower bound is in the optimal order by examples in \cite{courant,jenssen} mentioned earlier. The lower bound on density achieved in this paper can give us more precise estimate on the life span of classical solution than before and more importantly can motivate us in searching the lower bound of density for BV solutions including shock waves, which is a major obstacle in establishing large BV existence theory for Euler equations. The proof of Theorem \ref{thm1} mainly relies on the observation of a new invariant domain on some gradient variables, given in Theorem \ref{thm2} which itself is a very interesting result.

This paper is divided into 4 sections. In Section 2, we compare the old idea in \cite{G9} and the new idea in this paper. In Section 3, we introduce some basic setup and existing results. In Section 4, we prove the main theorem for the full Euler equations. 

\section{Idea of the proof for Theorem \ref{thm1}}
Now we briefly introduce the old ideas in \cite{G9} and then the new idea in this paper.

First, we recall two gradient variables used in earlier papers \cite{G3, G9,G6}\footnote{Note in \cite{G3, G9,G6}, variables $\alpha$ and $\beta$ mean  $\tilde\alpha$ and $\tilde\beta$ in this paper, respectively. In this paper, we reserve $\alpha$ and $\beta$ for other use.}
\beq \left.\begin{array}{l}
\tilde\alpha=u_x+m\eta_x+\frac{\gamma-1}{\gamma}m_x
\eta,\end{array}\right.\label{al_def0}\eeq \beq
\left.\begin{array}{l}\tilde\beta=u_x-m\eta_x-
\frac{\gamma-1}{\gamma}m_x \eta,\end{array}\right.\label{be_def0} \eeq
where two new variables $m$ and $\eta$ take the role of $S$ and $\tau$, respectively,
\beq\label{etam}
   m=e^{\frac{S}{2c_v}},\quad
   \eta  = \TS\frac{2\sqrt{K\gamma}}{\gamma-1}\,
\tau^{-\frac{\gamma-1}{2}}\,.
\eeq
The positive constants $K$ and $c_v$ are used in \eqref{introduction 3}. Using $\tilde\alpha$ and $\tilde\beta$, one can define the rarefaction and compression ($R/C$) characters in the local sense for non-isentropic solutions, under the following definition.
\begin{definition}
\label{def2}
\cite{G3,G6}
The local {$R/C$} character for a classical ($C^1$) solution of \eqref{lagrangian1}-\eqref{lagrangian3} is
\[\begin{array}{lllll} 
	\text{Forward}& $R$ &\text{iff} \quad\tilde\alpha>0 ,\\ 
	\text{Forward}& $C$ & \text{iff} \quad \tilde\alpha<0,\\
	\text{Backward}& $R$ & \text{iff} \quad \tilde\beta>0,\\ 
	\text{Backward}& $C$ &\text{iff}\quad \tilde\beta<0.
\end{array}\]
\end{definition}
%This local definition of rarefaction and compression characters generalizes the definition for isentropic solution.

To achieve a time dependent lower bound on density in the optimal order $O(1+t)^{-1}$,  the key step is to find a constant uniform upper bound on $\tilde\alpha(x,t)$ and $\tilde\beta(x,t)$ for any $(x,t)$ in the domain we consider, or in another word, to bound the maximum rarefaction (expansion) in both forward and backward directions.  Actually, if we can find a constant global upper bound on $\tilde\alpha(x,t)$ and $\tilde\beta(x,t)$, then by the conservation of mass (\ref{lagrangian1}),
we can easily prove that
\beq\label{main0}
\tau_t=u_x=\frac{1}{2}(\tilde\alpha+\tilde\beta)<\hbox{Constant},
\eeq
which directly gives the $O(1+t)^{-1}$ lower bound on density, together with the constant positive lower bound on the initial density. The relation \eqref{main0} was first noticed and used by \cite{G9}.

For the isentropic solutions, in \cite{G3}, we found that $\max_{x\in\mathbb R}\{\tilde\alpha(x,t),\tilde\beta(x,t)\}< N$ is invariant on $t$ when $N\geq0$, by studying the Riccati system given by Lax in \cite{lax2} that $\tilde\alpha$ and $\tilde\beta$ satisfy. Physically, this can be explained as that the maximum expansion will not increase with respect to time, although the backward expansion $\max_{x\in\mathbb R}\tilde\beta(x,t)$ or forward expansion $\max_{x\in\mathbb R}\tilde\alpha(x,t)$ might increase. Using this result, one can easily find constant upper bounds on $\tilde\alpha$ and $\tilde\beta$, then derive an $O(1+t)^{-1}$ lower bound on density by the argument in the previous paragraph. 

However, it is very hard to prove a similar decay for non-isentropic solutions. Some numerical evidences showed the possible increase of $\max_{x\in\mathbb R}\{\tilde\alpha(x,t),\tilde\beta(x,t)\}$ on $t$. Alternately, in \cite{G9} we proved that if
\beq\label{al-be-0}
\max_{x\in\mathbb R}\Big\{\rho^\ve\tilde\alpha(x,0),\rho^\ve\tilde\beta(x,0)\Big\}<N
\eeq
then 
\beq\label{al-be-t}
\max_{x\in\mathbb R}\Big\{\rho^\ve\tilde\alpha(x,t),\rho^\ve\tilde\beta(x,t)\Big\}<N
\eeq
for sufficiently small positive $\ve$ and sufficiently large $N$ depending on $\ve$ and the initial data, when $t>0$.
Although the use of $\rho^\ve$ helped us obtain some decay, this result could only provide an $O(1+t)^{-1-\delta}$ lower bound on density with $\delta=\frac{\ve}{1-\ve}>0$, i.e. an almost optimal order estimate.

In this paper, we will prove a lower bound on density in the optimal order $O(1+t)^{-1}$ for the non-isentropic solutions.
The key new idea is to consider the following gradient variables, transformed from $\tilde\alpha$ and $\tilde\beta$ in \eqref{al_def0}-\eqref{be_def0}:
\beq \left.\begin{array}{l}
\alpha:=\tilde{\alpha}+\lambda\, \eta\end{array}\right.\label{al2}\eeq \beq
\left.\begin{array}{l}\beta:=\tilde{\beta}+\lambda\, \eta,\end{array}\right.\label{be2} \eeq
for some $\lambda$ given later in \eqref{la_Def} with $\eta$ defined in \eqref{etam}, then show the following decay on $\max_{x\in\mathbb R}\{\alpha(x,t),\beta(x,t)\}$. 
\begin{theorem}\label{thm2}
For any smooth solution considered in Theorem \ref{thm1} satisfying the same initial assumptions, if 
\beq\label{ab0}
\max_{x\in\mathbb R}\{\alpha(x,0),\beta(x,0)\}<M\,,
\eeq
then 
\beq\label{abt}
\max_{x\in\mathbb R}\{\alpha(x,t),\beta(x,t)\}<M\,,
\eeq
for any $t\in \mathbb{R}^+$, where $M$ is a constant satisfying
\beq\label{M_def}
\textstyle M\geq M^*:=4 M_\eta \,M_D\cdot\max\left\{\frac{2(\gamma-1)}{\gamma},\frac{(3\gamma-1)}{\gamma+1}\frac{4}{\gamma}\right\}
\eeq
and
\beq\label{la_Def}
\textstyle \lambda:=\frac{1}{4M_\eta}\frac{\gamma+1}{3\gamma-1} M^*=\frac{\gamma+1}{3\gamma-1}M_D\cdot\max\left\{\frac{2(\gamma-1)}{\gamma},\frac{(3\gamma-1)}{\gamma+1}\frac{4}{\gamma}\right\}.
 \eeq
Here $M_D$ and $M_\eta$ are constant global upper bounds on $|m_x|$ and $\eta$ in \eqref{u_rho_bounds}
and \eqref{m_bounds}, respectively.
\end{theorem}
Using this result, we can find constant upper bounds on  $\tilde\alpha(x,t)$ and $\tilde\beta(x,t)$
by the constant global upper bound on $\eta$ given in \cite{G8}, although $\max_{x\in\mathbb R}\{\tilde\alpha(x,t),\tilde\beta(x,t)\}$
might not decay with respect to $t$. Then we can prove Theorem \ref{thm1} using \eqref{main0} and the initial constant lower bound on density (uniform upper bound on $\tau(x,0)$).

\section{Some basic setup and existing results for Euler equations}
We first introduce some notations and basic setups for $C^1$ solutions
of full Euler equations \eqref{lagrangian1}$\sim$\eqref{introduction 3}, which were used in \cite{G3}.

Recall we use new variables $m$ and $\eta$ to take the roles of $S$ and $\tau$, respectively:
\beq
   m=e^{\frac{S}{2c_v}}\label{m def}
\eeq
and 
\beq
   \eta  = \TS\frac{2\sqrt{K\gamma}}{\gamma-1}\,
\tau^{-\frac{\gamma-1}{2}}\,.\label{z def}
\eeq
We use $c$ to denote the nonlinear Lagrangian wave speed for full Euler equations, where
\beq
  c=\sqrt{-p_\tau}=
  \sqrt{K\,\gamma}\,{\tau}^{-\frac{\gamma+1}{2}}\,e^{\frac{S}{2c_v}}\,.
\label{c_def2}
\eeq
% Similarly, in this section, if we use the same name as in p-system to denote a function, this function is extended form the corresponding function in p-system from isentropic case to general case, which should have no ambiguity.
The forward and backward characteristics are described by
\beq\label{pmc_full}
  \frac{dx}{dt}=c \com{and} \frac{dx}{dt}=-c\,,
\eeq
and we denote the corresponding directional derivatives along these characteristics by
\[
  \pp := \dbyd t+c\,\dbyd x \com{and}
  \pn := \dbyd t-c\,\dbyd x\,,
\]
respectively.

It follows that
\begin{align}
  \tau&=K_{\tau}\,\eta^{-\frac{2}{\gamma-1}}\,,\nn\\
  p&=K_p\, m^2\, \eta^{\frac{2\gamma}{\gamma-1}}\,,\label{tau p c}\\
  c&=c(\eta,m)=K_c\, m\, \eta^{\frac{\gamma+1}{\gamma-1}}\,,\nn
\end{align}
with positive constants
\beq
  K_\tau:=\Big(\frac{2\sqrt{K\gamma}}{\gamma-1}\Big)^\frac{2}{\gamma-1}\,,
\quad
  K_p:=K\,K_\tau^{-\gamma},\com{and}
  K_c:=\TS\sqrt{K\gamma}\,K_\tau^{-\frac{\gamma+1}{2}}\,.
\label{Kdefs}
\eeq
%so that also
%\beq
%   K_p=\TS\frac{\gamma-1}{2\gamma}K_c \com{and}
%   K_\tau K_c=\frac{\gamma-1}{2}\,.
%\label{KpKcRela}
%\eeq

In these coordinates, for $C^1$ solutions, equations
\eq{lagrangian1}--\eqref{introduction 3} are equivalent to
\begin{align}
  \eta_t+\frac{c}{m}\,u_x&=0\,, \label{lagrangian1 zm}\\
  u_t+m\,c\,\eta_x+2\frac{p}{m}\,m_x&=0\,,\label{lagrangian2 zm}\\
  m_t&=0\label{lagrangian3 zm}\,,
\end{align}
where the last equation comes from (\ref{s con}), which is equivalent to
(\ref{lagrangian3}), c.f. \cite{smoller}.  Note that, while the solution remains $C^1$,
$m=m(x)$ is given by the initial data and can be regarded as a
stationary quantity.

We denote the Riemann variables by
\beq
  r:=u-m\,\eta\quad\hbox{and}\quad s:=u+m\,\eta\,.
\label{r_s_def}
\eeq
Different from the isentropic case ($m$ constant), for general non-isentropic flow, $s$ and $r$
vary along characteristics.  

Recall that we denote gradient variables $\tilde\alpha$ and $\tilde\beta$
%\beq \left.\begin{array}{l}
%\tilde\alpha=u_x+m\eta_x+\frac{\gamma-1}{\gamma}m_x
%\eta,\end{array}\right.\label{al_def}\eeq \beq
%\left.\begin{array}{l}\tilde\beta=u_x-m\eta_x-
%\frac{\gamma-1}{\gamma}m_x \eta,\end{array}\right.\label{be_def} \eeq
in \eqref{al_def0}-\eqref{be_def0}. These variables satisfy the following Riccati equations, where the detailed derivation  can be found in  \cite{G3}.
%Here $\alpha$ and
%$\beta$ are the generalization of $s_x$ and $r_x$ in a smoothly
%varying entropy field, where $s$, $r$ are Riemann invariants in a
%constant entropy field. 
{\begin{proposition}\cite{G3}\label{remark2} The classical solutions
for (\ref{lagrangian1})$\sim$(\ref{lagrangian3}) satisfy \beq
\left.\begin{array}{l}\partial_+\tilde\alpha=k_1\{k_2
(3\tilde\alpha+\tilde\beta)+\tilde\alpha\tilde\beta-\tilde\alpha^2\}, \label{frem1}
\end{array}\right.\eeq
and \beq \left.\begin{array}{l}\partial_-\tilde\beta=k_1\{-k_2
(\tilde\alpha+3\tilde\beta)+\tilde\alpha\tilde\beta-\tilde\beta^2\},\label{frem2}
\end{array}\right.\eeq
where \beq\left.\begin{array}{l}
k_1=\frac{(\gamma+1)K_c}{2(\gamma-1)} \eta^{\frac{2}{\gamma-1}}, \quad
k_2=\frac{\gamma-1}{\gamma(\gamma+1)}\eta\, m_x. \label{k def}\end{array}\right.\eeq
\end{proposition}

%%%%%%%%%%
%
%%%%%%%%%%

Finally, we review a result on the uniform upper bounds of $|u|$ and $\rho$  established by the author, Young and Zhang in \cite{G8}, for later references.

We always assume that all initial conditions in Theorem \ref{thm1} are satisfied.
By \eq{m def}, there exist constants $M_L$, $M_U$ and $M_D$ such that
\beq
  0 < M_L < m(\cdot) < M_U\,,\quad \hbox{and} \quad |m_x(\cdot)|< M_D\,.
\label{m_bounds}
\eeq
Also there exist positive constants $M_{\bar s}$ and $M_{\bar r}$,
such that, in the initial data,
\beq
  |s(\cdot,0)|<M_{\bar s} \com{and}
  |r(\cdot,0)|<M_{\bar r}\,.
\label{Mrs}
\eeq
%In this section, we always assume \eqref{Vdef}$\sim$\eqref{Mrs}.
 
In the following proposition established in \cite{G8} (Theorem 2.1), $|u|$ and $\eta$ are shown to be uniformly bounded above.

\begin{proposition}{\em \cite{G8}}
\label{Thm_upper}Assume all initial conditions in Theorem \ref{thm1} are satisfied.
And assume system \eqref{lagrangian1}$\sim$\eqref{lagrangian3} with \eqref{introduction 3}  has a  $C^1$ solution when $t\in[0,T)$, then one
has the uniform bounds
\beq\label{u_rho_bounds}
   |u(x,t)|\leq\frac{L_1+L_2}{2}{M_U}^{\frac{1}{2\gamma}}
\com{and}
   \eta(x,t)\leq\frac{L_1+L_2}{2}{M_L}^{\frac{1}{2\gamma}-1}=:M_\eta,
\eeq
where $L_1$ and $L_2$ are positive constants only depending on $\gamma$, $M_{\bar s}$, $M_{\bar r}$, $M_L$, $M_U$ and $V$
%\begin{align*}
%  L_1 &:= M_s+\ol V\,M_r+\ol V\,(\ol V\,M_s+{\ol V}^2\,M_r)
%	\,e^{{\ol V}^2},\\
%  L_2 &:= M_r+\ol V\,M_s+\ol V\,(\ol V\,M_r+{\ol V}^2\,M_s)
%	\,e^{{\ol V}^2},
%\end{align*}
with
\beq
  V := \frac{1}{2c_v}\int_{-\infty}^{+\infty}|S'(x)|\;dx
     = \int_{-\infty}^{+\infty}\frac{|m'(x)|}{m(x)}\;dx<\infty\,,
\label{Vdef}
\eeq
%and
%\[\ol V := \frac{V}{2\gamma}\,.
%\] 
%Constants $L_1$ and $L_2$ both clearly depend only on the
%initial data and $\gamma$.
Here $T$ can be any positive number or infinity. Two bounds in \eqref{u_rho_bounds} are both independent of $T$.
\end{proposition}
We remark that there is a typo in Theorem 2.1 in \cite{G8}, where $\rho$ shall be $\eta$.
Here $V<\infty$ since $S(x)$ is $C^1$ and has a finite BV norm.

%Clearly, if  the initial conditions in Theorem \ref{thm2} are satisfied, then conditions \eqref{Vdef}$\sim$\eqref{Mrs} are all satisfied, hence $\rho$ and $c$ are both uniformly bounded above.
%%%%%%%%%%
%
%%%%%%%%%%
\section{Proof of the main theorem}

Let's first prove the following key lemma. Recall that $\alpha$ and $\beta$ are defined in \eqref{al2} and \eqref{be2}, respectively.

\begin{lemma}\label{lem1}
Under the initial assumptions in Theorem \ref{thm1}, for any smooth solution in $t\in[0,T)$, if 
 $\frac{M}{2}\leq \alpha\leq M$ and $\beta\leq M$ on a piece of forward characteristic $\Gamma(t)$ with $t\in[t_2,t_1]\in[0,T)$, then there exists a positive constant $K_1$ only depending on $M_\eta$ and $\gamma>1$, such that,
\begin{align}
\partial_+\alpha\leq K_1M(M-\alpha),\label{4.10}
\end{align}
on $\Gamma(t)$ with $t\in[t_2,t_1]$. 

\end{lemma}
\begin{proof}
First by \eqref{lagrangian1 zm} and \eqref{al_def0}-\eqref{be_def0}, we have
\[
 \partial_+ \eta=\eta_t+c\eta_x=-\frac{c}{m}u_x+c\eta_x=
 -K_c\eta^{\frac{\gamma+1}{\gamma-1}}(\tilde\beta
+\textstyle\frac{\gamma-1}{\gamma}\eta m_x) \,,
\]
and
\[
 \partial_- \eta=\eta_t-c\eta_x=-\frac{c}{m}u_x-c\eta_x=-K_c\eta^{\frac{\gamma+1}{\gamma-1}}(\tilde\alpha
 -\textstyle\frac{\gamma-1}{\gamma}m_x )\,,
\]
then using (\ref{frem1}), we have
\begin{align}
&\partial_+\alpha\nn\\
=&\partial_+\tilde\alpha+\lambda\partial_+\eta \nn\\
=&\textstyle\frac{1}{2\gamma}K_c\eta^{\frac{\gamma+1}{\gamma-1}}m_x(3\tilde\alpha+\tilde\beta)
+\frac{\gamma+1}{2(\gamma-1)}K_c \eta^{\frac{2}{\gamma-1}}\tilde\alpha(\tilde\beta-\tilde\alpha)
-\lambda K_c \eta^{\frac{\gamma+1}{\gamma-1}} (\tilde\beta+\frac{\gamma-1}{\gamma}\eta m_x)\nn\\
=&\textstyle\frac{1}{2\gamma}K_c\eta^{\frac{\gamma+1}{\gamma-1}}m_x(4\alpha-4\lambda\eta +\beta-\alpha)
-\lambda K_c \eta^{\frac{\gamma+1}{\gamma-1}} (\beta-\alpha+\frac{\gamma-1}{\gamma}\eta m_x-\lambda \eta+\alpha)\nn\\
&\textstyle+\frac{\gamma+1}{2(\gamma-1)}K_c \eta^{\frac{2}{\gamma-1}}(\alpha-\lambda \eta)(\beta-\alpha)\nn\\
=&\textstyle\frac{2}{\gamma}K_c\eta^{\frac{\gamma+1}{\gamma-1}}m_x(\alpha-\lambda\eta)
-\lambda K_c\eta^{\frac{\gamma+1}{\gamma-1}}(\alpha-\lambda\eta+\frac{\gamma-1}{\gamma}\eta m_x)\nn\\
&\textstyle
+\frac{\gamma+1}{2(\gamma-1)}K_c\eta^{\frac{2}{\gamma-1}}(\alpha-\frac{3\gamma-1}{\gamma+1}\lambda\eta+\frac{\gamma-1}{\gamma(\gamma_+1)}\eta m_x)(\beta-\alpha)\nn\\
=&\textstyle\frac{2}{\gamma}K_c\eta^{\frac{\gamma+1}{\gamma-1}}m_x(\alpha-\lambda\eta)
-\frac{1}{2}\lambda K_c\eta^{\frac{\gamma+1}{\gamma-1}}(\alpha-\lambda\eta)-
\frac{1}{2}\lambda K_c\eta^{\frac{\gamma+1}{\gamma-1}}(\alpha-\lambda\eta+\frac{2(\gamma-1)}{\gamma}\eta m_x)\nn\\
&\textstyle
+\frac{\gamma+1}{2(\gamma-1)}K_c\eta^{\frac{2}{\gamma-1}}(\alpha-\frac{3\gamma-1}{\gamma+1}\lambda\eta+\frac{\gamma-1}{\gamma(\gamma_+1)}\eta m_x)(\beta-\alpha)\nn\\
=&\textstyle-\frac{1}{2}\lambda K_c\eta^{\frac{\gamma+1}{\gamma-1}}(\alpha-\lambda\eta+\frac{2(\gamma-1)}{\gamma}\eta m_x)-
\frac{1}{2}K_c\eta^{\frac{\gamma+1}{\gamma-1}}(\lambda-\frac{4}{\gamma }m_x)(\alpha-\lambda\eta)\nn\\
&\textstyle
+\frac{\gamma+1}{2(\gamma-1)}K_c\eta^{\frac{2}{\gamma-1}}(\alpha-\frac{3\gamma-1}{\gamma+1}\lambda\eta+\frac{\gamma-1}{\gamma(\gamma_+1)}\eta m_x)(\beta-\alpha).
\label{new1}
\end{align}

By \eqref{M_def} and \eqref{la_Def}, it is easy to see that, for any $\gamma>1$,
\beq\label{3.12}
\textstyle \lambda\geq 
 \frac{4}{\gamma}M_D,
 \eeq
\beq\label{3.13}
\textstyle\frac{1}{4}M\geq \frac{2(\gamma-1)}{\gamma}M_\eta \,M_D=
\max\left\{\frac{2(\gamma-1)}{\gamma},\frac{\gamma-1}{\gamma(\gamma+1)}\right\}M_\eta \,M_D,
\eeq
and
\beq\label{3.14}
\textstyle\frac{1}{4}M\geq\frac{3\gamma-1}{\gamma+1} M_\eta\lambda=\max\left\{\lambda, \frac{3\gamma-1}{\gamma+1}\lambda \right\}M_\eta,
\eeq
since $\textstyle\frac{2(\gamma-1)}{\gamma}>\frac{\gamma-1}{\gamma(\gamma+1)}$ and
$\textstyle \frac{3\gamma-1}{\gamma+1}>1$ when $\gamma>1$. Recall that $M_D$ and $M_\eta$ are upper bounds on $|m_x|$
and $\eta$, respectively.

Recall that on the forward characteristic piece $\Gamma(t)$ with $t\in[t_2,t_1]$, 
\beq\label{4.4}\frac{M}{2}\leq \alpha\leq M,\qquad \beta\leq M. \eeq
So by \eqref{3.12}-\eqref{3.14}, we know that
 \beq\textstyle\alpha-\lambda\eta+\frac{2(\gamma-1)}{\gamma}\eta m_x>0,\quad 
(\lambda-\frac{4}{\gamma }m_x)(\alpha-\lambda\eta)>0\eeq and
\beq\label{4.6}\textstyle 0<\alpha-\frac{3\gamma-1}{\gamma+1}\lambda\eta+\frac{\gamma-1}{\gamma(\gamma_+1)}\eta m_x
\leq \frac{5}{4}M.\eeq
Then using  inequalities in \eqref{4.4}-\eqref{4.6} on \eqref{new1}, we have
\begin{align}
\partial_+\alpha\nn
\leq& \textstyle\frac{\gamma+1}{2(\gamma-1)}K_c\eta^{\frac{2}{\gamma-1}}(\alpha-\frac{3\gamma-1}{\gamma+1}\lambda\eta+\frac{\gamma-1}{\gamma(\gamma_+1)}\eta m_x)(M-\alpha)\nn\\
\leq&\textstyle\frac{\gamma+1}{2(\gamma-1)}K_c\eta^{\frac{2}{\gamma-1}} \frac{5}{4}M(M-\alpha)\nn\\
\leq& K_1M(M-\alpha)\label{4.11},
\end{align}
for some constant $K_1>0$ only depending on $M_\eta$ and $\gamma>1$.
\end{proof}

Similar as \eqref{new1}, $\beta$ satisfies
\begin{align}
\partial_-\beta\nn=&\textstyle-\frac{1}{2}\lambda K_c\eta^{\frac{\gamma+1}{\gamma-1}}(\beta-\lambda\eta-\frac{2(\gamma-1)}{\gamma}\eta m_x)-
\frac{1}{2}K_c\eta^{\frac{\gamma+1}{\gamma-1}}(\lambda+\frac{4}{\gamma }m_x)(\beta-\lambda\eta)\nn\\
&\textstyle
+\frac{\gamma+1}{2(\gamma-1)}K_c\eta^{\frac{2}{\gamma-1}}(\beta-\frac{3\gamma-1}{\gamma+1}\lambda\eta-\frac{\gamma-1}{\gamma(\gamma_+1)}\eta m_x)(\alpha-\beta).
\label{new2}
\end{align}
So by \eqref{3.12}-\eqref{3.14}, a symmetric result as in Lemma \ref{lem1} holds on the backward characteristic direction. Here we omit the detail.
\smallskip

Next we will prove Theorem \ref{thm2}. Let's first explain the key idea.
We observe that the domain 
$\max_{x\in\mathbb R}\{\alpha(x,t),\beta(x,t)\}<M$ is invariant on $t$. In fact, as shown in Figure \ref{pic}, it is easy to check that $\partial_+\alpha<0$ on the right boundary by \eqref{4.11} (strict inequality when $\beta<M$), and similarly $\partial_-\beta<0$ on the upper boundary, of the region $\max\{\alpha,\beta\}<M$, except at the point $(\alpha,\beta)=(M,M)$. So the solution will not flow out of this domain if it is initially in it. 

	\begin{figure}[htp] \centering
		\includegraphics[width=.3\textwidth]{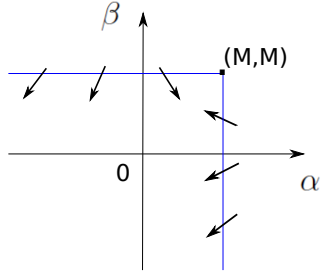}
		\caption{Classical solutions will not flow out of the domain $\max\{\alpha,\beta\}<M$.
		\label{pic}}
	\end{figure} 
\bigskip

Now we give the proof of Theorem \ref{thm2}.
\begin{proof}
We first prove \eqref{abt} by contradiction. Without loss of generality, assume that 
$\max\{\alpha,\beta\}(x_0,t_0)=M$ at some point $(x_0,t_0)$. See Figure \ref{fig0}.

	\begin{figure}[htp] \centering
		\includegraphics[width=.4\textwidth]{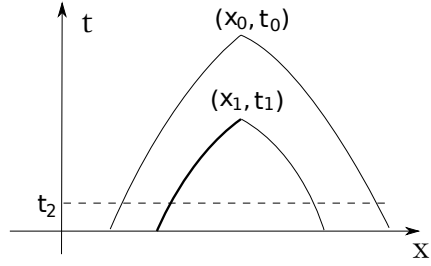}
		\caption{Proof of Theorem \ref{thm2}.\label{fig0}}
	\end{figure}

Because wave speed $c$ is bounded above, we can find the characteristic triangle with vertex $(x_0,t_0)$ and lower boundary
on the initial line $t=0$, denoted by $\Omega$, which is the outer characteristic triangle in Figure \ref{fig0}. 

In the closed region $\Omega$, we can find the first time $t_1$ such that $\alpha=M$ or $\beta=M$ in $\Omega$. More precisely, There exists a point $(x_1,t_1)\in\Omega$ such that, $\alpha(x_1,t_1)=M$ or/and $\beta(x_1,t_1)=M$, while
\[
\max_{(x,t)\in \Omega,\, t<t_1}\Big\{\alpha(x,t),\beta(x,t)\Big\}<M.\]

Without loss of generality, assume 
$\alpha(x_1,t_1)=M$ and $\beta(x_1,t_1)\leq M$. The proof for another case when $\beta(x_1,t_1)=M$ is entirely same.
Let's denote the characteristic triangle with vertex $(x_1,t_1)$ as  $\Omega_1\in\Omega$, where $\Omega_1$ is the inner characteristic triangle in Figure \ref{fig0}. Then
\beq\label{p_main_1}
\max_{(x,t)\in \Omega_1}\Big\{\alpha(x,t),\beta(x,t)\Big\}<M,\eeq
and $\alpha(x_1,t_1)=M$.
By the continuity of $\alpha$,  we could find a time $t_2\in[0,t_1)$ such that,
\beq\label{p_main_2}
\alpha(x,t)\geq M/2\,,\quad \hbox{for any}\quad (x,t)\in \Omega_1\quad \hbox{and}\quad t\geq t_2\,.
\eeq

Next we derive a contradiction. By Lemma \ref{lem1} and 
\eqref{p_main_1}$\sim$\eqref{p_main_2}, along the forward characteristic segment  through $(x_1,t_1)$ when $t_2\leq t<t_1$,
\[
\partial_+\alpha\leq K_1 M(M-\alpha).
\]
which gives, through the integration along the forward characteristic,
\[
\ln \frac{1}{M-\alpha(t)}\leq\textstyle K_1 M(t-t_2)\,.
\]
As $t\rightarrow t_1-$, the left hand side approaches infinity 
while the right hand side approaches a finite number, which gives a contradiction. Hence we prove \eqref{abt} then the Theorem \ref{thm2}.
\end{proof}

Finally, we prove Theorem \ref{thm1}.

First, using the initial bounds on $\alpha(x,0)$ and $\beta(x,0)$ from the bounds on  $\tilde\alpha(x,0)$ and $\tilde\beta(x,0)$ and $\eta(x,0)$, we can find a constant $M$ for \eqref{ab0}. Then using Theorem \ref{thm2}, we know that
$\alpha(x,t)$ and $\beta(x,t)$ are also uniformly bounded by $M$. Then,
because of the constant uniform upper bound on $\lambda\, \eta$ by Proposition \ref{Thm_upper}, we can directly prove that $\tilde\alpha(x,t)$ and $\tilde\beta(x,t)$ are also uniformly bounded above by a constant in the domain $(x,t)\in \mathbb R\times[0,T)$, using definitions \eqref{al2}-\eqref{be2}. Then by  \eqref{main0} and the initial constant lower bound on  $\rho(x,0)$ (uniform upper bound on $\tau(x,0)$)  for any $x\in \mathbb{R}$, we can prove \eqref{max}. So Theorem \ref{thm1} is proved.

\section*{Acknowledgement}
The author's research  was partially supported by NSF, with grant  DMS-1715012.}

%\bibitem{gej} G. Chen, Erik Endres and Helge Kristan Jenssen,
%{\it Pairwise wave interactions in ideal polytropic gases}, Arch. Rat. Mech. Anal. {204:3} (2012), 787-836.

\end{document}